\documentclass[leqno,12pt]{article} %leqno is the option to put formula numbers on the left side
\usepackage{amsmath, amssymb}
\usepackage{amsthm} %theorem environment option
\theoremstyle{plain} %text of this environment is typesetted in italics
\newtheorem{theorem}{\indent\sc Theorem}[section]
\newtheorem{lemma}[theorem]{\indent\sc Lemma}
\newtheorem{corollary}[theorem]{\indent\sc Corollary}
\newtheorem{proposition}[theorem]{\indent\sc Proposition}

\theoremstyle{definition} %text of this environment is typesetted in roman letters

\newtheorem{remark}[theorem]{\indent\sc Remark}
\newtheorem{example}[theorem]{\indent\sc Example}

\title{On certain symmetries of $\mathbb R^3$ with a diagonal metric}
\author{Adara M. Blaga}

\date{}
\begin{document}

\maketitle

\markboth{{\small\it {\hspace{1cm} On certain symmetries of $\mathbb R^3$ with a diagonal metric}}}{\small\it{On certain symmetries of $\mathbb R^3$ with a diagonal metric\hspace{1cm}}}

%%%%%%%%%%%%%%% footnote %%%%%%%%%%%%%%%%
\footnote{ %2010 MSC numbers
}
\footnote{ %key words and phrases
}

\begin{abstract}
We determine Killing vector fields on the $3$-dimensional space $\mathbb R^3$ endowed with a special diagonal metric.
\end{abstract}

\section{Preliminaries}

The aim of the paper is to determine Killing vector fields on the $3$-dimensional space $\mathbb R^3$ endowed with some special diagonal metrics, extending the results for the $2$-dimensional case treated in \cite{bl24}. Due to the fact that determining the Killing vector fields for a general diagonal metric is an extended study, we begin it here and we continue it, for other types of Lam\'{e} coefficients, in a different work.

Let ${g}$ be a Riemannian metric on $\mathbb R^3$ given by
$${g}=\frac{1}{f_1^2}dx^1\otimes dx^1+\frac{1}{f_2^2}dx^2\otimes dx^2+\frac{1}{f_3^2}dx^3\otimes dx^3,$$
where $f_1$, $f_2$ and $f_3$ are smooth functions nowhere zero on $\mathbb R^3$, and $x^1,x^2,x^3$ stand for the standard coordinates in $\mathbb R^3$. Let
$$\Big\{E_1:=f_1\frac{\partial}{\partial x^1}, \ \ E_2:=f_2\frac{\partial}{\partial x^2}, \ \ E_3:=f_3\frac{\partial}{\partial x^3}\Big\}$$
be a local orthonormal frame. We will denote as follows:
$$\frac{f_2}{f_1}\cdot\frac{\partial f_1}{\partial x^2}=:f_{12}, \ \ \frac{f_3}{f_1}\cdot\frac{\partial f_1}{\partial x^3}=:f_{13}, \ \ \frac{f_1}{f_2}\cdot\frac{\partial f_2}{\partial x^1}=:f_{21},$$$$\frac{f_3}{f_2}\cdot\frac{\partial f_2}{\partial x^3}=:f_{23}, \ \ \frac{f_1}{f_3}\cdot\frac{\partial f_3}{\partial x^1}=:f_{31}, \ \ \frac{f_2}{f_3}\cdot\frac{\partial f_3}{\partial x^2}=:f_{32}.$$
The Levi-Civita connection $\nabla$ of $g$ is given by (see \cite{balr}):
$$\nabla_{E_1}E_1=f_{12}E_2+f_{13}E_3, \ \ \nabla_{E_2}E_2=f_{21}E_1+f_{23}E_3, \ \ \nabla_{E_3}E_3=f_{31}E_1+f_{32}E_2,$$
$$\nabla_{E_1}E_2=-f_{12}E_1, \ \ \nabla_{E_2}E_3=-f_{23}E_2, \ \ \nabla_{E_3}E_1=-f_{31}E_3,$$
$$\nabla_{E_1}E_3=-f_{13}E_1, \ \ \nabla_{E_3}E_2=-f_{32}E_3, \ \ \nabla_{E_2}E_1=-f_{21}E_2.$$

In the rest of the paper, whenever a function $f$ depends only on some of its variables, we will write in its argument only that variables in order to emphasize this fact, for example, $f(x^i)$, $f(x^i,x^j)$.

\section{Killing vector fields w.r.t. these metrics}

We recall that a vector field $V$ on $(\mathbb R^3,g)$ is called a \textit{Killing vector field} \cite{killing} if the Lie derivative $\pounds$ of the metric $g$ in the direction of $V$ vanishes, i.e.,
\begin{equation*}
(\pounds_Vg)(X,Y):=V(g(X,Y))-g([V,X],Y)-g(X,[V,Y])=0
\end{equation*}
for any smooth vector fields $X$ and $Y$ on $\mathbb R^3$.

In certain particular cases, we shall determine the Killing vector fields, as well as the relation between two Killing vector fields for the metric $g$.

Let $V=\sum_{k=1}^3V^kE_k$, with $V^k$, $k\in \{1,2,3\}$, smooth functions on $\mathbb R^3$. Then
\begin{align*}
(\pounds_Vg)(E_i,E_j)&=g(\nabla_{E_i}V,E_j)+g(E_i,\nabla_{E_j}V)\\
&=E_i(V^j)+E_j(V^i)+\sum_{k=1}^3V^k\{g(\nabla_{E_i}E_k,E_j)+g(E_i,\nabla_{E_j}E_k)\}
\end{align*}
for any $i,j\in \{1,2,3\}$, which is equivalent to
\begin{equation*}
\left\{
    \begin{aligned}
&(\pounds_Vg)(E_1,E_1)=2\{E_1(V^1)-f_{12}V^2-f_{13}V^3\}\\
&(\pounds_Vg)(E_2,E_2)=2\{E_2(V^2)-f_{21}V^1-f_{23}V^3\}\\
&(\pounds_Vg)(E_3,E_3)=2\{E_3(V^3)-f_{31}V^1-f_{32}V^2\}\\
&(\pounds_Vg)(E_1,E_2)=E_1(V^2)+E_2(V^1)+f_{12}V^1+f_{21}V^2\\
&(\pounds_Vg)(E_2,E_3)=E_2(V^3)+E_3(V^2)+f_{23}V^2+f_{32}V^3\\
&(\pounds_Vg)(E_3,E_1)=E_3(V^1)+E_1(V^3)+f_{31}V^3+f_{13}V^1
    \end{aligned}
  \right. \ ,
\end{equation*}
and we can state:

\begin{proposition}
The vector field $V=\sum_{k=1}^3V^kE_k$ is a Killing vector field if and only if
\begin{equation}\label{s1}
\left\{
    \begin{aligned}
&f_1\frac{\displaystyle \partial V^1}{\displaystyle \partial x^1}-\frac{\displaystyle f_2}{\displaystyle f_1}\cdot\frac{\displaystyle \partial f_1}{\displaystyle \partial x^2}V^2-\frac{\displaystyle f_3}{\displaystyle f_1}\cdot\frac{\displaystyle \partial f_1}{\displaystyle \partial x^3}V^3=0\\
&f_2\frac{\displaystyle \partial V^2}{\displaystyle \partial x^2}-\frac{\displaystyle f_1}{\displaystyle f_2}\cdot\frac{\displaystyle \partial f_2}{\displaystyle \partial x^1}V^1-\frac{\displaystyle f_3}{\displaystyle f_2}\cdot\frac{\displaystyle \partial f_2}{\displaystyle \partial x^3}V^3=0\\
&f_3\frac{\displaystyle \partial V^3}{\displaystyle \partial x^3}-\frac{\displaystyle f_1}{\displaystyle f_3}\cdot\frac{\displaystyle \partial f_3}{\displaystyle \partial x^1}V^1-\frac{\displaystyle f_2}{\displaystyle f_3}\cdot\frac{\displaystyle \partial f_3}{\displaystyle \partial x^2}V^2=0\\
&f_1\frac{\displaystyle \partial V^2}{\displaystyle \partial x^1}+f_2\frac{\displaystyle \partial V^1}{\displaystyle \partial x^2}+\frac{\displaystyle f_2}{\displaystyle f_1}\cdot\frac{\displaystyle \partial f_1}{\displaystyle \partial x^2}V^1+\frac{\displaystyle f_1}{\displaystyle f_2}\cdot\frac{\displaystyle \partial f_2}{\displaystyle \partial x^1}V^2=0\\
&f_2\frac{\displaystyle \partial V^3}{\displaystyle \partial x^2}+f_3\frac{\displaystyle \partial V^2}{\displaystyle \partial x^3}+\frac{\displaystyle f_3}{\displaystyle f_2}\cdot\frac{\displaystyle \partial f_2}{\displaystyle \partial x^3}V^2+\frac{\displaystyle f_2}{\displaystyle f_3}\cdot\frac{\displaystyle \partial f_3}{\displaystyle \partial x^2}V^3=0\\
&f_3\frac{\displaystyle \partial V^1}{\displaystyle \partial x^3}+f_1\frac{\displaystyle \partial V^3}{\displaystyle \partial x^1}+\frac{\displaystyle f_1}{\displaystyle f_3}\cdot\frac{\displaystyle \partial f_3}{\displaystyle \partial x^1}V^3+\frac{\displaystyle f_3}{\displaystyle f_1}\cdot\frac{\displaystyle \partial f_1}{\displaystyle \partial x^3}V^1=0
    \end{aligned}
  \right. \ .
\end{equation}
\end{proposition}

A natural question is: \textit{For which functions $f_1$, $f_2$ and $f_3$, the basis vector fields $E_1$, $E_2$ and $E_3$ are Killing vector fields w.r.t. $g$?} And we can state:
\begin{proposition}\label{asd}
$E_1$ is a Killing vector field on $(\mathbb R^3,g)$ if and only if
\begin{equation*}
\left\{
    \begin{aligned}
&f_1=f_1(x^1)\\
&f_2=f_2(x^2,x^3)\\
&f_3=f_3(x^2,x^3)
    \end{aligned}
  \right. \ .
\end{equation*}
\end{proposition}
\begin{proof}
Replacing $V^1=1$, $V^2=V^3=0$ in (\ref{s1}), we get
\begin{equation*}
\left\{
    \begin{aligned}
&\frac{\displaystyle \partial f_2}{\displaystyle \partial x^1}=0\\
&\frac{\displaystyle \partial f_3}{\displaystyle \partial x^1}=0\\
&\frac{\displaystyle \partial f_1}{\displaystyle \partial x^2}=0\\
&\frac{\displaystyle \partial f_1}{\displaystyle \partial x^3}=0
    \end{aligned}
  \right. \ ,
\end{equation*}
hence, we get the conclusion.
\end{proof}

\begin{example}
The vector field
$E_1=e^{x^1}\displaystyle \frac{\partial }{\partial x^1}$
is a Killing vector field on
\begin{equation*}
\left(\mathbb R^3, \ g=e^{-2x^1}dx^1\otimes dx^1+e^{x^2+x^3}dx^2\otimes dx^2+e^{x^2x^3}dx^3\otimes dx^3\right).
\end{equation*}
\end{example}

Similarly, for the other two vector fields $E_2$ and $E_3$, we have:
\begin{corollary}
(i) $E_2$ is a Killing vector field on $(\mathbb R^3,g)$ if and only if
\begin{equation*}
\left\{
    \begin{aligned}
&f_1=f_1(x^1,x^3)\\
&f_2=f_2(x^2)\\
&f_3=f_3(x^1,x^3)
    \end{aligned}
  \right. \ .
\end{equation*}

(ii) $E_3$ is a Killing vector field on $(\mathbb R^3,g)$ if and only if
\begin{equation*}
\left\{
    \begin{aligned}
&f_1=f_1(x^1,x^2)\\
&f_2=f_2(x^1,x^2)\\
&f_3=f_3(x^3)
    \end{aligned}
  \right. \ .
\end{equation*}
\end{corollary}

\begin{lemma}
If $f_i=f_i(x^1)$ for any $i\in \{1,2,3\}$, then $V=\sum_{k=1}^3V^kE_k$ is a Killing vector field if and only if
\begin{equation}\label{s2}
\left\{
    \begin{aligned}
&\frac{\displaystyle \partial V^1}{\displaystyle \partial x^1}=0\\
&f_2\frac{\displaystyle \partial V^2}{\displaystyle \partial x^2}-f_1\frac{\displaystyle f_2'}{\displaystyle f_2}V^1=0\\
&f_3\frac{\displaystyle \partial V^3}{\displaystyle \partial x^3}-f_1\frac{\displaystyle f_3'}{\displaystyle f_3}V^1=0\\
&f_1\frac{\displaystyle \partial V^2}{\displaystyle \partial x^1}+f_2\frac{\displaystyle \partial V^1}{\displaystyle \partial x^2}+f_1\frac{\displaystyle f_2'}{\displaystyle f_2}V^2=0\\
&f_2\frac{\displaystyle \partial V^3}{\displaystyle \partial x^2}+f_3\frac{\displaystyle \partial V^2}{\displaystyle \partial x^3}=0\\
&f_3\frac{\displaystyle \partial V^1}{\displaystyle \partial x^3}+f_1\frac{\displaystyle \partial V^3}{\displaystyle \partial x^1}+f_1\frac{\displaystyle f_3'}{\displaystyle f_3}V^3=0
    \end{aligned}
  \right. \ .
\end{equation}
\end{lemma}
\begin{proof}
It follows immediately from \eqref{s1}.
\end{proof}

\begin{theorem}\label{te1}
If $f_1=f_1(x^1)$, $f_2=f_2(x^1)$, $f_3=k_3\in \mathbb R\setminus\{0\}$, then $V=\sum_{k=1}^3V^kE_k$ is a Killing vector field if and only if one of the following assertions hold:

\hspace{0.5cm}\begin{equation*}
(i) \
\left\{
    \begin{aligned}
&V^1=0\\
&V^2(x^1)=\displaystyle\frac{c_1}{f_2(x^1)}, \ \ c_1\in \mathbb R\\
&V^3=c_2, \ \ c_2\in \mathbb R
\end{aligned}
  \right. \ ;
\end{equation*}

\hspace{0.5cm}
\begin{equation*}
(ii) \
\left\{
    \begin{aligned}
&V^1(x^2,x^3)=c_1x^2+c_2x^3+c_3, \ \ c_1,c_2,c_3\in \mathbb R\\
&V^2(x^1,x^3)=-c_1k_2F(x^1)-c_4k_2x^3+c_5, \ \ c_4,c_5\in \mathbb R\\
&V^3(x^1,x^2)=-c_2k_3F(x^1)+c_4k_3x^2+c_6, \ \ c_6\in \mathbb R
\end{aligned}
  \right. \ ,
\end{equation*}
where $F'=\displaystyle\frac{1}{f_1}$ and $f_2=k_2$ (is constant);

\hspace{0.5cm}
\begin{equation*}
(iii) \
\left\{
    \begin{aligned}
&V^1(x^2)=c_1x^2+c_2, \ \ c_1, c_2\in \mathbb R\\
&V^2(x^1,x^2)=\left(f_1\frac{f_2'}{f_2^2}\right)(x^1)\left[\displaystyle\frac{c_1}{2}(x^2)^2+c_2x^2+c_3\right]+\displaystyle \frac{c_1F_0(x^1)+c_4}{f_2(x^1)}, \ \ c_3, c_4\in \mathbb R\\
&V^3=c_3
\end{aligned}
  \right. \ ,
\end{equation*}
with $\displaystyle\frac{f_1'}{f_1}\cdot\displaystyle\frac{f_2'}{f_2}+\left(\displaystyle\frac{f_2'}{f_2}\right)'=0$ and $f_2$ nonconstant, where $F_0'=-\displaystyle\frac{f_2^2}{f_1}$;

\hspace{0.5cm}
\begin{equation*}
(iv) \ \left\{
    \begin{aligned}
&V^1(x^2)=c_1\cos({\sqrt{k}x^2})+c_2\sin({\sqrt{k}x^2}), \ \ c_1, c_2\in \mathbb R\\
&V^2(x^1,x^2)=\left(f_1\frac{f_2'}{f_2^2}\right)(x^1)\left[\displaystyle\frac{1}{\sqrt{k}}[c_1\sin({\sqrt{k}x^2})-c_2\cos({\sqrt{k}x^2})]+c_3\right]\\
&\hspace{72pt}+\displaystyle \frac{c_3kF_0(x^1)+c_4}{f_2(x^1)}, \ \ c_3, c_4\in \mathbb R\\
&V^3=c_3
\end{aligned}
  \right. \ ,
\end{equation*}
with $k:=\left(\displaystyle\frac{f_1}{f_2}\right)^2\left[\displaystyle\frac{f_1'}{f_1}\cdot\displaystyle\frac{f_2'}{f_2}+\left(\displaystyle\frac{f_2'}{f_2}\right)'\right]\in (0,+\infty)$, where $F_0'=-\displaystyle\frac{f_2^2}{f_1}$;

\hspace{0.5cm}
\begin{equation*}
(v) \
\left\{
    \begin{aligned}
&V^1(x^2)=c_1e^{\sqrt{-k}x^2}+c_2e^{-\sqrt{-k}x^2}, \ \ c_1, c_2\in \mathbb R\\
&V^2(x^1,x^2)=\left(f_1\frac{f_2'}{f_2^2}\right)(x^1)\left[\displaystyle\frac{1}{\sqrt{-k}}\left(c_1e^{\sqrt{-k}x^2}-c_2e^{-\sqrt{-k}x^2}\right)+c_3\right]\\
&\hspace{72pt}+\displaystyle \frac{c_3kF_0(x^1)+c_4}{f_2(x^1)}, \ \ c_3, c_4\in \mathbb R\\
&V^3=c_3
\end{aligned}
  \right. \ ,
\end{equation*}
with $k:=\left(\displaystyle\frac{f_1}{f_2}\right)^2\left[\displaystyle\frac{f_1'}{f_1}\cdot\displaystyle\frac{f_2'}{f_2}+\left(\displaystyle\frac{f_2'}{f_2}\right)'\right]\in (-\infty,0)$, where $F_0'=-\displaystyle\frac{f_2^2}{f_1}$.
\end{theorem}
\begin{proof}
In this case, we have
\begin{equation}\label{s2d0}
\left\{
    \begin{aligned}
&\frac{\displaystyle \partial V^1}{\displaystyle \partial x^1}=0\\
&f_2\frac{\displaystyle \partial V^2}{\displaystyle \partial x^2}-f_1\frac{\displaystyle f_2'}{\displaystyle f_2}V^1=0\\
&\frac{\displaystyle \partial V^3}{\displaystyle \partial x^3}=0\\
&f_1\frac{\displaystyle \partial V^2}{\displaystyle \partial x^1}+f_2\frac{\displaystyle \partial V^1}{\displaystyle \partial x^2}+f_1\frac{\displaystyle f_2'}{\displaystyle f_2}V^2=0\\
&f_2\frac{\displaystyle \partial V^3}{\displaystyle \partial x^2}+k_3\frac{\displaystyle \partial V^2}{\displaystyle \partial x^3}=0\\
&k_3\frac{\displaystyle \partial V^1}{\displaystyle \partial x^3}+f_1\frac{\displaystyle \partial V^3}{\displaystyle \partial x^1}=0
    \end{aligned}
  \right. \ .
\end{equation}
From the first and the third equations of (\ref{s2d0}), we get that
\begin{equation*}
V^1=V^1(x^2,x^3), \ \ V^3=V^3(x^1,x^2),
\end{equation*}
thus, from the last equation, we deduce that
\begin{equation*}
\left\{
    \begin{aligned}
&\frac{\displaystyle \partial V^1}{\displaystyle \partial x^3}=-\frac{\displaystyle 1}{\displaystyle k_3}F_{2}\\
&\frac{\displaystyle \partial V^3}{\displaystyle \partial x^1}=\frac{\displaystyle 1}{\displaystyle f_1}F_{2}
   \end{aligned}
  \right. \ ,
\end{equation*}
where $F_{2}=F_{2}(x^2)$, which, by integration, give
\begin{equation}\label{s1bc10}
\left\{
    \begin{aligned}
&V^1(x^2,x^3)=-\frac{\displaystyle x^3}{\displaystyle k_3}F_{2}(x^2)+G_1(x^2)\\
&V^3(x^1,x^2)=F_{2}(x^2)F(x^1)+G_3(x^2)
\end{aligned}
  \right. \ ,
\end{equation}
where $F'=\displaystyle\frac{1}{f_1}$, $G_1=G_1(x^2)$, $G_3=G_3(x^2)$. Then
\begin{equation}\label{s1bd10}
f_1(x^1)\frac{\displaystyle \partial V^2}{\displaystyle \partial x^1}(x^1,x^2,x^3)=-f_2(x^1)\left[-\frac{\displaystyle x^3}{\displaystyle k_3}F'_{2}(x^2)+G'_1(x^2)\right]-\left(f_1\frac{\displaystyle f_2'}{\displaystyle f_2}\right)(x^1)V^2(x^1,x^2,x^3),
\end{equation}
\begin{equation}\label{s1bd11}
k_3\frac{\displaystyle \partial V^2}{\displaystyle \partial x^3}(x^1,x^2,x^3)=-f_2(x^1)\left[F(x^1)F'_{2}(x^2)+G'_3(x^2)\right].
\end{equation}
By differentiating the above equations with respect to $x^2$, we get
\begin{equation}\label{s1be10}
f_1(x^1)\frac{\displaystyle \partial^2 V^2}{\displaystyle \partial x^1\partial x^2}(x^1,x^2,x^3)=-f_2(x^1)\left[-\frac{\displaystyle x^3}{\displaystyle k_3}F''_{2}(x^2)+G''_1(x^2)\right]-\left(f_1\frac{\displaystyle f_2'}{\displaystyle f_2}\right)(x^1)\frac{\displaystyle \partial V^2}{\displaystyle \partial x^2}(x^1,x^2,x^3),
\end{equation}
\begin{equation}\label{s1be11}
k_3\frac{\displaystyle \partial^2 V^2}{\displaystyle \partial x^3\partial x^2}(x^1,x^2,x^3)=-f_2(x^1)\left[F(x^1)F''_{2}(x^2)+G''_3(x^2)\right].
\end{equation}
Replacing $V^1$ in the second equation of (\ref{s2d0}), we find
\begin{equation}\label{ada6}
\frac{\displaystyle \partial V^2}{\displaystyle \partial x^2}(x^1,x^2,x^3)=\left(f_1\frac{\displaystyle f_2'}{\displaystyle f_2^2}\right)(x^1)\left[-\frac{\displaystyle x^3}{\displaystyle k_3}F_{2}(x^2)+G_1(x^2)\right],
\end{equation}
which, replaced in (\ref{s1be10}) and (\ref{s1be11}), gives
\begin{equation}\label{s1be103}
\left(\frac{\displaystyle f_1}{\displaystyle f_2}\left(f_1\frac{\displaystyle f_2'}{\displaystyle f_2^2}\right)'+\left(f_1\frac{\displaystyle f_2'}{\displaystyle f_2^2}\right)^2\right)(x^1)\left[-\frac{\displaystyle x^3}{\displaystyle k_3}F_{2}(x^2)+G_1(x^2)\right]=-\left[-\frac{\displaystyle x^3}{\displaystyle k_3}F''_{2}(x^2)+G''_1(x^2)\right]
\end{equation}
and
\begin{equation}\label{s1be113}
\left(f_1\frac{\displaystyle f_2'}{\displaystyle f_2^3}\right)(x^1)F_{2}(x^2)=F(x^1)F''_{2}(x^2)+G''_3(x^2).
\end{equation}
Let us denote $h=\frac{\displaystyle f_1}{\displaystyle f_2}\left(f_1\frac{\displaystyle f_2'}{\displaystyle f_2^2}\right)'+\left(f_1\frac{\displaystyle f_2'}{\displaystyle f_2^2}\right)^2$.
By differentiating (\ref{s1be103}) with respect to $x^3$ and $x^1$, we successively infer
\begin{equation}\label{s1be104}
h(x^1)F_{2}(x^2)=-F''_{2}(x^2)
\end{equation}
and
\begin{equation*}
h'(x^1)F_{2}(x^2)=0,
\end{equation*}
which is equivalent to
\begin{equation}\label{ba}
h=k_1\in \mathbb R \ \ \textrm{or} \ \ F_2=0.
\end{equation}
Let us denote $l=f_1\left(f_1\frac{\displaystyle f_2'}{\displaystyle f_2^3}\right)'$.
Two times differentiating (\ref{s1be113}) with respect to $x^1$, we get
\begin{equation*}
l(x^1)F_{2}(x^2)=F''_{2}(x^2)
\end{equation*}
and
\begin{equation*}
l'(x^1)F_{2}(x^2)=0,
\end{equation*}
which is equivalent to
\begin{equation}\label{ca}
l=k_2\in \mathbb R \ \ \textrm{or} \ \ F_2=0.
\end{equation}
From (\ref{ba}) and (\ref{ca}), we conclude that
\begin{equation*}
F_2=0 \ \ \textrm{or} \ \
\left\{
    \begin{aligned}
&\frac{\displaystyle f_1}{\displaystyle f_2}\left(f_1\frac{\displaystyle f_2'}{\displaystyle f_2^2}\right)'+\left(f_1\frac{\displaystyle f_2'}{\displaystyle f_2^2}\right)^2=k_1\in \mathbb R\\
&f_1\left(f_1\frac{\displaystyle f_2'}{\displaystyle f_2^3}\right)'=k_2\in \mathbb R
\end{aligned}
  \right. \ .
\end{equation*}

(A) Let $F_2=0$. Then (\ref{s1bc10}) implies that
\begin{equation}\label{frou}
\left\{
    \begin{aligned}
&V^1(x^2)=G_1(x^2)\\
&V^3(x^2)=G_3(x^2)
\end{aligned}
  \right. \ .
\end{equation}
Also:
\begin{equation*}
\left\{
    \begin{aligned}
&\frac{\displaystyle \partial V^2}{\displaystyle \partial x^1}(x^1,x^2,x^3)=-\left(\frac{\displaystyle f_2}{\displaystyle f_1}\right)(x^1)G_1'(x^2)-\left(\frac{\displaystyle f_2'}{\displaystyle f_2}\right)(x^1)V^2(x^1,x^2,x^3)\\
&\frac{\displaystyle \partial V^2}{\displaystyle \partial x^2}(x^1,x^2,x^3)=\left(f_1\frac{\displaystyle f_2'}{\displaystyle f_2^2}\right)(x^1)G_1(x^2)\\
&\frac{\displaystyle \partial V^2}{\displaystyle \partial x^3}(x^1,x^2,x^3)=-\frac{\displaystyle f_2(x^1)}{\displaystyle k_3}G_3'(x^2)
\end{aligned}
  \right. \ .
\end{equation*}
By integrating the last equation, we get
\begin{equation}\label{frau}
V^2(x^1,x^2,x^3)=-\frac{\displaystyle f_2(x^1)}{\displaystyle k_3}G_3'(x^2)x^3+H(x^1,x^2),
\end{equation}
where $H=H(x^1,x^2)$, which, by differentiating with respect to $x^1$ and $x^2$ respectively, and considering the first two equations of the previous system, give
\begin{equation}\label{pd}
-2\frac{\displaystyle f_2'(x^1)}{\displaystyle k_3}G_3'(x^2)x^3=-\left(\displaystyle\frac{f_2}{f_1}\right)(x^1)G_1'(x^2)-\left(\displaystyle\frac{f_2'}{f_2}\right)(x^1)H(x^1,x^2)-
\displaystyle\frac{\partial H}{\partial x^1}(x^1,x^2)
\end{equation}
and
\begin{equation}\label{pe}
-\frac{\displaystyle f_2(x^1)}{\displaystyle k_3}G_3''(x^2)x^3=\left(f_1\displaystyle\frac{f_2'}{f_2^2}\right)(x^1)G_1(x^2)-
\displaystyle\frac{\partial H}{\partial x^2}(x^1,x^2).
\end{equation}
By differentiating (\ref{pd}) and (\ref{pe}) with respect to $x^3$, we find
\begin{equation*}
f_2'(x^1)G_3'(x^2)=0
\end{equation*}
and
\begin{equation*}
G_3''=0,
\end{equation*}
which is equivalent to
\begin{equation}\label{pe1}
f_2=c_2\in \mathbb R\setminus\{0\} \ \ \textrm{or} \ \ G_3=c_3\in \mathbb R
\end{equation}
and
\begin{equation}\label{pe2}
G_3(x^2)=a_1x^2+b_1, \ \ a_1,b_1\in \mathbb R.
\end{equation}
From (\ref{pe1}) and (\ref{pe2}), we conclude that
\begin{equation*}
\left\{
    \begin{aligned}
&f_2=c_2\in \mathbb R\setminus\{0\}\\
&G_3(x^2)=a_1x^2+b_1, \ \ a_1, b_1\in \mathbb R
\end{aligned}
  \right.
 \ \ \textrm{or} \ \
G_3=c_3\in \mathbb R.
\end{equation*}

\hspace{0.5cm} (A.1) Let \begin{equation*}
\left\{
    \begin{aligned}
&f_2=c_2\in \mathbb R\setminus\{0\}\\
&G_3(x^2)=a_1x^2+b_1, \ \ a_1, b_1\in \mathbb R
\end{aligned}
  \right. \ .
\end{equation*}
Then
\begin{equation*}
\left\{
    \begin{aligned}
&V^1(x^2)=G_1(x^2)\\
&V^2(x^1,x^2,x^3)=-\displaystyle\frac{a_1c_2}{k_3}x^3+H(x^1,x^2)\\
&V^3(x^2)=a_1x^2+b_1
\end{aligned}
  \right. \ .
\end{equation*}
From (\ref{pd}) and (\ref{pe}), we get
\begin{equation*}
\left\{
    \begin{aligned}
&\displaystyle\frac{\partial H}{\partial x^1}(x^1,x^2)=-\displaystyle\frac{c_2}{f_1(x^1)}G_1'(x^2)\\
&\displaystyle\frac{\partial H}{\partial x^2}(x^1,x^2)=0
\end{aligned}
  \right. \ .
\end{equation*}
From the second equation we deduce that $H=H(x^1)$, and from the first one, we deduce that
\begin{equation*}
-\frac{(f_1H')(x^1)}{c_2}=G_1'(x^2)
\end{equation*}
must be constant, let's say, $c_0\in \mathbb R$, and we obtain
\begin{equation*}
H(x^1)=-c_0c_2F(x^1)+c_4, \ \ c_4\in \mathbb R
\end{equation*}
and
\begin{equation*}
G_1(x^2)=c_0x^2+c_5, \ \ c_5\in \mathbb R.
\end{equation*}
Therefore,
\begin{equation*}
\left\{
    \begin{aligned}
&V^1(x^2)=c_0x^2+c_5\\
&V^2(x^1,x^3)=-\displaystyle\frac{a_1c_2}{k_3}x^3-c_0c_2F(x^1)+c_4\\
&V^3(x^2)=a_1x^2+b_1
\end{aligned}
  \right. \ .
\end{equation*}

\hspace{0.5cm} (A.2) Let $G_3=c_3\in \mathbb R$.
Then (\ref{frou}) and (\ref{frau}) imply that
\begin{equation*}
\left\{
    \begin{aligned}
&V^1(x^2)=G_1(x^2)\\
&V^2(x^1,x^2)=H(x^1,x^2)\\
&V^3=c_3
\end{aligned}
  \right. \ .
\end{equation*}
From (\ref{pd}) and (\ref{pe}), we get
\begin{equation}\label{pl1}
\left\{
    \begin{aligned}
&\displaystyle\frac{\partial H}{\partial x^1}(x^1,x^2)=-\left(\displaystyle\frac{f_2'}{f_2}\right)(x^1)H(x^1,x^2)-\left(\displaystyle\frac{f_2}{f_1}\right)(x^1)G_1'(x^2)\\
&\displaystyle\frac{\partial H}{\partial x^2}(x^1,x^2)=\left(f_1\displaystyle\frac{f_2'}{f_2^2}\right)(x^1)G_1(x^2)
\end{aligned}
  \right. \ .
\end{equation}
The first equation of the previous system is equivalent to
\begin{equation*}
\displaystyle\frac{\partial }{\partial x^1}\left(f_2H(\cdot,x^2)\right)(x^1)=-\left(\displaystyle\frac{f_2^2}{f_1}\right)(x^1)G_1'(x^2),
\end{equation*}
which, by integration, gives
\begin{equation*}
H(x^1,x^2)=\left(\frac{F_0}{f_2}\right)(x^1)G_1'(x^2)+\displaystyle\frac{1}{f_2(x^1)}K(x^2),
\end{equation*}
where $F_0'=-\displaystyle\frac{f_2^2}{f_1}$ and $K=K(x^2)$. Differentiating it with respect to $x^2$ and taking into account the second equation of (\ref{pl1}), we find
\begin{equation*}
F_0(x^1)G_1''(x^2)+K'(x^2)=\left(f_1\displaystyle\frac{f_2'}{f_2}\right)(x^1)G_1(x^2),
\end{equation*}
which, by differentiating with respect to $x^1$, implies
\begin{equation*}
G_1''(x^2)=-\left(\displaystyle\frac{f_1}{f_2^2}\right)(x^1)\left(f_1\displaystyle\frac{f_2'}{f_2}\right)'(x^1)G_1(x^2)=-h(x^1)G_1(x^2),
\end{equation*}
where
\begin{equation*}
h=\frac{\displaystyle f_1}{\displaystyle f_2}\left(f_1\frac{\displaystyle f_2'}{\displaystyle f_2^2}\right)'+\left(f_1\frac{\displaystyle f_2'}{\displaystyle f_2^2}\right)^2=\left(\frac{f_1}{f_2}\right)^2\left[\frac{f_1'}{f_1}\cdot\frac{f_2'}{f_2}+\left(\frac{f_2'}{f_2}\right)'\right],
\end{equation*}
which further, by differentiating with respect to $x^1$, implies
\begin{equation*}
h'(x^1)G_1(x^2)=0,
\end{equation*}
which is equivalent to
\begin{equation*}
h=k_1\in \mathbb R \ \ \textrm{or} \ \ G_1=0.
\end{equation*}

\hspace{1cm} (A.2.1) Let $h=k_1\in \mathbb R$.
We have the following cases:

\hspace{1.5cm} (A.2.1.1) $k_1=0$; in this case, $G_1(x^2)=c_1x^2+c_2$, $c_1,c_2\in \mathbb R$;

\hspace{1.5cm} (A.2.1.2) $k_1> 0$; in this case, $G_1(x^2)=c_1\cos({\sqrt{k_1}x^2})+c_2\sin({\sqrt{k_1}x^2})$, $c_1,c_2\in \nolinebreak \mathbb R$;

\hspace{1.5cm} (A.2.1.3) $k_1< 0$; in this case, $G_1(x^2)=c_1e^{\sqrt{-k_1}x^2}+c_2e^{-\sqrt{-k_1}x^2}$, $c_1,c_2\in \mathbb R$.\\
By integrating the second equation of (\ref{pl1}), we get
\begin{equation*}
H(x^1,x^2)=\left(f_1\frac{f_2'}{f_2^2}\right)(x^1)G_0(x^2)+L(x^1),
\end{equation*}
where $G_0'=G_1$ and $L=L(x^1)$. Differentiating it with respect to $x^1$ and taking into account the first equation of (\ref{pl1}), we find
\begin{equation*}
\left(f_1\frac{f_2'}{f_2^2}\right)'(x^1)G_0(x^2)+L'(x^1)=
-\left(\displaystyle\frac{f_2'}{f_2}\right)(x^1)\left[\left(f_1\frac{f_2'}{f_2^2}\right)(x^1)G_0(x^2)+L(x^1)\right]-
$$$$-\left(\displaystyle\frac{f_2}{f_1}\right)(x^1)G_0''(x^2),
\end{equation*}
which is equivalent to
\begin{equation*}
G_0''(x^2)+k_1G_0(x^2)=
-\left(\displaystyle\frac{f_1}{f_2^2}(f_2L)'\right)(x^1),
\end{equation*}
which must be a constant, let's say, $k\in \mathbb R$, and we obtain
\begin{equation*}
L(x^1)=\displaystyle \frac{kF_0(x^1)+k_0}{f_2(x^1)},
\end{equation*}
where $k_0\in \mathbb R$, and
\begin{equation*}
G_0''(x^2)+k_1G_0(x^2)=k.
\end{equation*}
Therefore,

\hspace{1.5cm} (A.2.1.1) if $k_1=0$:
\begin{equation*}
\left\{
    \begin{aligned}
&G_1(x^2)=c_1x^2+c_2, \ \ c_1,c_2\in \mathbb R\\
&G_0(x^2)=\displaystyle\frac{c_1}{2}(x^2)^2+c_2x^2+c_3, \ \ c_3\in \mathbb R\\
&k=c_1
\end{aligned}
  \right. \ .
\end{equation*}
In this case,
\begin{equation*}
\left\{
    \begin{aligned}
&V^1(x^2)=c_1x^2+c_2\\
&V^2(x^1,x^2)=\left(f_1\frac{f_2'}{f_2^2}\right)(x^1)\left[\displaystyle\frac{c_1}{2}(x^2)^2+c_2x^2+c_3\right]+\displaystyle \frac{c_1F_0(x^1)+k_0}{f_2(x^1)}\\
&V^3=c_3
\end{aligned}
  \right. \ ,
\end{equation*}
where $F_0'=-\displaystyle\frac{f_2^2}{f_1}$;

\hspace{1.5cm} (A.2.1.2) if $k_1> 0$:
\begin{equation*}
\left\{
    \begin{aligned}
&G_1(x^2)=c_1\cos({\sqrt{k_1}x^2})+c_2\sin({\sqrt{k_1}x^2}), \ \ c_1,c_2\in \mathbb R\\
&G_0(x^2)=\displaystyle\frac{1}{\sqrt{k_1}}[c_1\sin({\sqrt{k_1}x^2})-c_2\cos({\sqrt{k_1}x^2})]+c_3, \ \ c_3\in \mathbb R\\
&k=c_3k_1
\end{aligned}
  \right. \ .
\end{equation*}
In this case,
\begin{equation*}
\left\{
    \begin{aligned}
&V^1(x^2)=c_1\cos({\sqrt{k_1}x^2})+c_2\sin({\sqrt{k_1}x^2})\\
&V^2(x^1,x^2)=\left(f_1\frac{f_2'}{f_2^2}\right)(x^1)\left[\displaystyle\frac{1}{\sqrt{k_1}}[c_1\sin({\sqrt{k_1}x^2})-c_2\cos({\sqrt{k_1}x^2})]+c_3\right]\\
&\hspace{72pt}+\displaystyle \frac{c_3k_1F_0(x^1)+k_0}{f_2(x^1)}\\
&V^3=c_3
\end{aligned}
  \right. \ ,
\end{equation*}
where $F_0'=-\displaystyle\frac{f_2^2}{f_1}$;

\hspace{1.5cm} (A.2.1.3) if $k_1< 0$:
\begin{equation*}
\left\{
    \begin{aligned}
&G_1(x^2)=c_1e^{\sqrt{-k_1}x^2}+c_2e^{-\sqrt{-k_1}x^2}, \ \ c_1,c_2\in \mathbb R\\
&G_0(x^2)=\displaystyle\frac{1}{\sqrt{-k_1}}\left(c_1e^{\sqrt{-k_1}x^2}-c_2e^{-\sqrt{-k_1}x^2}\right)+c_3, \ \ c_3\in \mathbb R\\
&k=c_3k_1
\end{aligned}
  \right. \ .
\end{equation*}
In this case,
\begin{equation*}
\left\{
    \begin{aligned}
&V^1(x^2)=c_1e^{\sqrt{-k_1}x^2}+c_2e^{-\sqrt{-k_1}x^2}\\
&V^2(x^1,x^2)=\left(f_1\frac{f_2'}{f_2^2}\right)(x^1)\left[\displaystyle\frac{1}{\sqrt{-k_1}}\left(c_1e^{\sqrt{-k_1}x^2}-c_2e^{-\sqrt{-k_1}x^2}\right)+c_3\right]+\displaystyle \frac{c_3k_1F_0(x^1)+k_0}{f_2(x^1)}\\
&V^3=c_3
\end{aligned}
  \right. \ ,
\end{equation*}
where $F_0'=-\displaystyle\frac{f_2^2}{f_1}$.

\hspace{1cm} (A.2.2) Let $G_1=0$. Then
\begin{equation*}
\left\{
    \begin{aligned}
&V^1=0\\
&\displaystyle\frac{\partial H}{\partial x^1}(x^1,x^2)=-\left(\displaystyle\frac{f_2'}{f_2}\right)(x^1)H(x^1,x^2)\\
&\displaystyle\frac{\partial H}{\partial x^2}(x^1,x^2)=0
\end{aligned}
  \right. \ .
\end{equation*}
From the last equation we deduce that $H=H(x^1)$, and from the second one, that
\begin{equation*}
(f_2H)'=0,
\end{equation*}
that is,
\begin{equation*}
H(x^1)=\frac{c_0}{f_2(x^1)},
\end{equation*}
where $c_0\in \mathbb R$.
Therefore,
\begin{equation*}
\left\{
    \begin{aligned}
&V^1=0\\
&V^2(x^1)=\displaystyle\frac{c_0}{f_2(x^1)}\\
&V^3=c_3
\end{aligned}
  \right. \ .
\end{equation*}

(B) Let
\begin{equation*}
\left\{
    \begin{aligned}
&\frac{\displaystyle f_1}{\displaystyle f_2}\left(f_1\frac{\displaystyle f_2'}{\displaystyle f_2^2}\right)'+\left(f_1\frac{\displaystyle f_2'}{\displaystyle f_2^2}\right)^2=k_1\in \mathbb R\\
&f_1\left(f_1\frac{\displaystyle f_2'}{\displaystyle f_2^3}\right)'=k_2\in \mathbb R
\end{aligned}
  \right. \ .
\end{equation*}
By straightforward computations, we obtain
\begin{equation*}
\left(f_1\frac{\displaystyle f_2'}{\displaystyle f_2^2}\right)^2=\frac{k_1-k_2}{2};\end{equation*}
therefore, $f_1\frac{\displaystyle f_2'}{\displaystyle f_2^2}$ must be constant, let's say, $k_0\in \mathbb R$, and we get
\begin{equation*}
k_1=-k_2=k_0^2.
\end{equation*}
We have
\begin{equation*}
\frac{\displaystyle f_2'}{\displaystyle f_2^2}=\frac{\displaystyle k_0}{\displaystyle f_1},
\end{equation*}
which, by integration, gives
$f_2(x^1)=-\displaystyle \frac{1}{k_0F(x^1)-c_0}$ on an open interval $I\subset \mathbb R$, where $c_0\in \mathbb R$ such that $\frac{\displaystyle c_0}{\displaystyle k_0}\notin F(I)$.
Also, $h=k_0^2$ and from (\ref{s1be104}), we get
\begin{equation*}
k_0^2F_2(x^2)=-F_2''(x^2).
\end{equation*}
The associated characteristic equation is $y^2+k_0^2=0$, and we have the following cases.

\hspace{0.5cm} (B.1) If $k_0=0$ (equivalent to $f_2'=0$, i.e., $f_2=c_2\in \mathbb R\setminus \{0\}$), then
\begin{equation*}
F_2(x^2)=a_1x^2+a_2, \ \ a_1, a_2\in \mathbb R.
\end{equation*}
Then (\ref{s1bc10})--(\ref{s1bd11}) and (\ref{ada6}) imply that
\begin{equation*}
\left\{
    \begin{aligned}
&V^1(x^2,x^3)=-\frac{\displaystyle (a_1x^2+a_2)x^3}{\displaystyle k_3}+G_1(x^2)\\
&V^3(x^1,x^2)=(a_1x^2+a_2)F(x^1)+G_3(x^2)\\
&\frac{\displaystyle \partial V^2}{\displaystyle \partial x^1}(x^1,x^2,x^3)=-\frac{\displaystyle c_2}{\displaystyle f_1(x^1)}\left[-\frac{\displaystyle a_1x^3}{\displaystyle k_3}+G_1'(x^2)\right]\\
&\frac{\displaystyle \partial V^2}{\displaystyle \partial x^2}(x^1,x^2,x^3)=0\\
&\frac{\displaystyle \partial V^2}{\displaystyle \partial x^3}(x^1,x^2,x^3)=-\frac{\displaystyle c_2}{\displaystyle k_3}[a_1F(x^1)+G_3'(x^2)]
\end{aligned}
  \right. \ .
\end{equation*}
It follows that $V^2=V^2(x^1,x^3)$, and since $c_2\neq 0$, from the third and the last equations, we get
\begin{equation*}
G_1=d_1\in \mathbb R, \ \ G_3=d_3\in \mathbb R,
\end{equation*}
and the previous system becomes
\begin{equation*}
\left\{
    \begin{aligned}
&V^1(x^2,x^3)=-\frac{\displaystyle (a_1x^2+a_2)x^3}{\displaystyle k_3}+d_1\\
&V^3(x^1,x^2)=(a_1x^2+a_2)F(x^1)+d_3\\
&\frac{\displaystyle \partial V^2}{\displaystyle \partial x^1}(x^1,x^3)=\frac{\displaystyle a_1c_2}{\displaystyle k_3f_1(x^1)}x^3\\
&\frac{\displaystyle \partial V^2}{\displaystyle \partial x^3}(x^1,x^3)=-\frac{\displaystyle a_1c_2}{\displaystyle k_3}F(x^1)
\end{aligned}
  \right. \ .
\end{equation*}
By integrating the last equation, we get
\begin{equation*}
V^2(x^1,x^3)=-\frac{\displaystyle a_1c_2}{\displaystyle k_3}F(x^1)x^3+K(x^1),\end{equation*}
where $K=K(x^1)$. By differentiating this relation with respect to $x^1$ and using the third equation of the previous system, we find
\begin{equation*}
K'(x^1)=\frac{\displaystyle 2a_1c_2}{\displaystyle k_3f_1(x^1)}x^3;
\end{equation*}
therefore, $a_1=0$, hence,
\begin{equation*}
\left\{
    \begin{aligned}
&V^1(x^3)=-\frac{\displaystyle a_2x^3}{\displaystyle k_3}+d_1\\
&V^3(x^1)=a_2F(x^1)+d_3\\
&\frac{\displaystyle \partial V^2}{\displaystyle \partial x^1}(x^1,x^3)=0\\
&\frac{\displaystyle \partial V^2}{\displaystyle \partial x^3}(x^1,x^3)=0
\end{aligned}
  \right. \ .
\end{equation*}
Thus, $V^2$ must be constant.

\hspace{0.5cm} (B.2) If $k_0\neq 0$, then
\begin{equation*}
F_2(x^2)=a_1\cos({k_0x^2})+a_2\sin({k_0x^2}), \ \ a_1, a_2\in \mathbb R.
\end{equation*}
From (\ref{s1be103}) and (\ref{s1be113}), we get
 \begin{equation*}
\left\{
    \begin{aligned}
&k_0^2G_1(x^2)=-G_1''(x^2)\\
&G_3''(x^2)=\left[\frac{\displaystyle k_0}{\displaystyle f_2(x^1)}+k_0^2F(x^1)\right]F_2(x^2),
\end{aligned}
  \right.
\end{equation*}
which is equivalent to
\begin{equation}\label{aaa}
\left\{
    \begin{aligned}
&G_1(x^2)=b_1\cos({k_0x^2})+b_2\sin({k_0x^2}), \ \ b_1, b_2\in \mathbb R\\
&F_2=0\\
&G_3(x^2)=a_3x^2+a_4, \ \ a_3, a_4\in \mathbb R
\end{aligned}
  \right.
\end{equation}
or
\begin{equation}\label{laka}
\left\{
    \begin{aligned}
&G_1(x^2)=b_1\cos({k_0x^2})+b_2\sin({k_0x^2}), \ \ b_1, b_2\in \mathbb R\\
&\frac{\displaystyle k_0}{\displaystyle f_2(x^1)}+k_0^2F(x^1)=k\in \mathbb R\\
&G_3''(x^2)=kF_2(x^2)
\end{aligned}
  \right. \ .
\end{equation}

\hspace{1cm} (B.2.1)
If (\ref{aaa}) holds true, then (\ref{s1bd10}), (\ref{ada6}), and (\ref{s1bd11}) imply:
\begin{equation}\label{ada12}
\left\{
    \begin{aligned}
&\frac{\displaystyle \partial V^2}{\displaystyle \partial x^1}(x^1,x^2,x^3)=-\left(\frac{\displaystyle f_2}{\displaystyle f_1}\right)(x^1)G_1'(x^2)-\left(\frac{\displaystyle f_2'}{\displaystyle f_2}\right)(x^1)V^2(x^1,x^2,x^3)\\
&\frac{\displaystyle \partial V^2}{\displaystyle \partial x^2}(x^1,x^2,x^3)=k_0G_1(x^2)\\
&\frac{\displaystyle \partial V^2}{\displaystyle \partial x^3}(x^1,x^2,x^3)=-\frac{\displaystyle f_2(x^1)}{\displaystyle k_3}G_3'(x^2)
\end{aligned}
  \right. \ .
\end{equation}
By integrating the third equation of (\ref{ada12}), we get
\begin{equation}\label{ada1a}
V^2(x^1,x^2,x^3)=-\displaystyle \frac{a_3}{k_3}f_2(x^1)x^3+H(x^1,x^2),
\end{equation}
where $H=H(x^1,x^2)$.
Differentiating (\ref{ada1a}) with respect to $x^2$ and replacing it in the second equation of (\ref{ada12}), we find
\begin{equation}\label{ada2a}
\displaystyle \frac{\partial H}{\partial x^2}(x^1,x^2)=k_0G_1(x^2),
\end{equation}
and by differentiating (\ref{ada1a}) with respect to $x^1$ and replacing it in the first equation of (\ref{ada12}), we find
\begin{equation}\label{dar1}
   \begin{aligned}
-\displaystyle \frac{f_2'(x^1)}{k_3}G_3'(x^2)x^3+\displaystyle \frac{\partial H}{\partial x^1}(x^1,x^2)&=-\left(\frac{\displaystyle f_2}{\displaystyle f_1}\right)(x^1)G_1'(x^2)+\frac{\displaystyle f_2'(x^1)}{\displaystyle k_3}G_3'(x^2)x^3\\
&\hspace{24pt}-\left(\frac{\displaystyle f_2'}{\displaystyle f_2}\right)(x^1)H(x^1,x^2)
\end{aligned}
  \ .
\end{equation}
From (\ref{dar1}) and (\ref{ada2a}), we have
\begin{equation}\label{ad1a}
\left\{
    \begin{aligned}
&\displaystyle \frac{\partial H}{\partial x^1}(x^1,x^2)=\frac{\displaystyle 2}{\displaystyle k_3}f_2'(x^1)a_3x^3-\left(\frac{\displaystyle f_2}{\displaystyle f_1}\right)(x^1)G_1'(x^2)-\left(\frac{\displaystyle f_2'}{\displaystyle f_2}\right)(x^1)H(x^1,x^2)\\
&\displaystyle \frac{\partial H}{\partial x^2}(x^1,x^2)=k_0G_1(x^2)
\end{aligned}
  \right. \ .
\end{equation}
From the first equation of (\ref{ad1a}), we deduce that
$a_3=0$ since $f_2'\neq 0$ at every point.
By integrating the last equation of the same system, we get
\begin{equation*}
H(x^1,x^2)=b_1\sin({k_0x^2})-b_2\cos({k_0x^2})+K(x^1),
\end{equation*}
where $K=K(x^1)$, which, by differentiating with respect to $x^1$ and replaced in the first one, gives
\begin{align*}
(f_2K)'=0,
\end{align*}
and we obtain
\begin{equation*}
K(x^1)=\frac{c_9}{f_2(x^1)}, \ \ c_9\in \mathbb R.
\end{equation*}
In this case,
\begin{equation*}
\left\{
    \begin{aligned}
&V^1(x^2)=b_1\cos({k_0x^2})+b_2\sin({k_0x^2})\\
&V^2(x^1, x^2)=b_1\sin({k_0x^2})-b_2\cos({k_0x^2})+\frac{c_9}{f_2(x^1)}\\
&V^3=a_4
\end{aligned}
  \right. \ .
\end{equation*}

\hspace{1cm} (B.2.2) If (\ref{laka}) holds true, then from the last two equations of (\ref{laka}), we get
$f_2(x^1)=\frac{\displaystyle k_0}{\displaystyle k-k_0^2F(x^1)}$ on an open interval $I\subset \mathbb R$
such that $\frac{\displaystyle k}{\displaystyle k_0^2}\notin F(I)$, and
\begin{equation*}
G_3(x^2)=-\frac{k}{k_0^2}[a_1\cos({k_0x^2})+a_2\sin({k_0x^2})]+a_3x^2+a_4, \ \ a_3, a_4\in \mathbb R.
\end{equation*}
Also, (\ref{s1bd10}), (\ref{ada6}), and (\ref{s1bd11}) imply:
\begin{equation}\label{ada}
\left\{
    \begin{aligned}
&\frac{\displaystyle \partial V^2}{\displaystyle \partial x^1}(x^1,x^2,x^3)=-\left(\frac{\displaystyle f_2}{\displaystyle f_1}\right)(x^1)\left[-\frac{\displaystyle x^3}{\displaystyle k_3}F_2'(x^2)+G_1'(x^2)\right]-\left(\frac{\displaystyle f_2'}{\displaystyle f_2}\right)(x^1)V^2(x^1,x^2,x^3)\\
&\frac{\displaystyle \partial V^2}{\displaystyle \partial x^2}(x^1,x^2,x^3)=k_0\left[-\frac{\displaystyle x^3}{\displaystyle k_3}F_2(x^2)+G_1(x^2)\right]\\
&\frac{\displaystyle \partial V^2}{\displaystyle \partial x^3}(x^1,x^2,x^3)=-\frac{\displaystyle f_2(x^1)}{\displaystyle k_3}\left[F(x^1)F_2'(x^2)+G_3'(x^2)\right]
\end{aligned}
  \right. \ .
\end{equation}
By integrating the third equation of (\ref{ada}), we get
\begin{equation}\label{ada1}
V^2(x^1,x^2,x^3)=-\displaystyle \frac{f_2(x^1)}{k_3}\left[F(x^1)F_2'(x^2)+G_3'(x^2)\right]x^3+H(x^1,x^2),
\end{equation}
where $H=H(x^1,x^2)$.
Differentiating (\ref{ada1}) with respect to $x^2$ and replacing it in the second equation of (\ref{ada}), we find
\begin{equation}\label{ada2}
-\displaystyle \frac{f_2(x^1)}{k_3}\left[F(x^1)F_2''(x^2)+G_3''(x^2)\right]x^3+\displaystyle \frac{\partial H}{\partial x^2}(x^1,x^2)=k_0\left[-\frac{\displaystyle x^3}{\displaystyle k_3}F_2(x^2)+G_1(x^2)\right],
\end{equation}
and by differentiating (\ref{ada1}) with respect to $x^1$ and replacing it in the first equation of (\ref{ada}), we find
\begin{equation}\label{dar}
   \begin{aligned}
&-\displaystyle \frac{f_2'(x^1)}{k_3}\left[F(x^1)F_2'(x^2)+G_3'(x^2)\right]x^3-\displaystyle \frac{f_2(x^1)}{k_3f_1(x^1)}F_2'(x^2)x^3+\displaystyle \frac{\partial H}{\partial x^1}(x^1,x^2)\\
&\hspace{12pt}=\frac{\displaystyle f_2(x^1)}{\displaystyle k_3f_1(x^1)}F_2'(x^2)x^3-\left(\frac{\displaystyle f_2}{\displaystyle f_1}\right)(x^1)G_1'(x^2)+\frac{\displaystyle f_2'(x^1)}{\displaystyle k_3}\left[F(x^1)F_2'(x^2)+G_3'(x^2)\right]x^3\\
&\hspace{24pt}-\left(\frac{\displaystyle f_2'}{\displaystyle f_2}\right)(x^1)H(x^1,x^2)
\end{aligned}
  \ .
\end{equation}
From (\ref{dar}) and (\ref{ada2}), we have
\begin{equation*}
\left\{
    \begin{aligned}
&\displaystyle \frac{\partial H}{\partial x^1}(x^1,x^2)=\frac{\displaystyle 2}{\displaystyle k_3}\left[\left(\frac{\displaystyle f_2}{\displaystyle f_1}\right)(x^1)F_2'(x^2)+f_2'(x^1)\left[F(x^1)F_2'(x^2)+G_3'(x^2)\right]\right]x^3\\
&\hspace{72pt}-\left(\frac{\displaystyle f_2}{\displaystyle f_1}\right)(x^1)G_1'(x^2)-\left(\frac{\displaystyle f_2'}{\displaystyle f_2}\right)(x^1)H(x^1,x^2)\\
&\displaystyle \frac{\partial H}{\partial x^2}(x^1,x^2)=\displaystyle \frac{f_2(x^1)}{k_3}\left[F(x^1)F_2''(x^2)+G_3''(x^2)\right]x^3+k_0\left[-\frac{\displaystyle x^3}{\displaystyle k_3}F_2(x^2)+G_1(x^2)\right]
\end{aligned}
  \right. \ ,
\end{equation*}
and we deduce that
\begin{equation}\label{pel}
\left\{
    \begin{aligned}
&F_2'(x^2)+k_0f_2(x^1)\left[F(x^1)F_2'(x^2)+G_3'(x^2)\right]=0\\
&\displaystyle \frac{\partial H}{\partial x^1}(x^1,x^2)=
-\left(\frac{\displaystyle f_2}{\displaystyle f_1}\right)(x^1)G_1'(x^2)-\left(\frac{\displaystyle f_2'}{\displaystyle f_2}\right)(x^1)H(x^1,x^2)\\
&f_2(x^1)\left[F(x^1)F_2''(x^2)+G_3''(x^2)\right]-k_0F_2(x^2)=0\\
&\displaystyle \frac{\partial H}{\partial x^2}(x^1,x^2)=k_0G_1(x^2)
\end{aligned}
  \right. \ .
\end{equation}
Taking into account that
\begin{equation*}
\left\{
    \begin{aligned}
&F_2(x^2)=a_1\cos({k_0x^2})+a_2\sin({k_0x^2})\\
&G_1(x^2)=b_1\cos({k_0x^2})+b_2\sin({k_0x^2})\\
&G_3(x^2)=-\frac{k}{k_0^2}[a_1\cos({k_0x^2})+a_2\sin({k_0x^2})]+a_3x^2+a_4
\end{aligned}
  \right. \ ,
\end{equation*}
from the first equation of (\ref{pel}), we obtain
\begin{equation*}
[a_1\sin({k_0x^2})-a_2\cos({k_0x^2})][k_0+k_0^2(f_2F)(x^1)-kf_2(x^1)]=k_0a_3f_2(x^1),
\end{equation*}
which implies that $a_1=a_2=0$, hence,
$a_3=0$, and
\begin{equation*}
\left\{
    \begin{aligned}
&F_2=0\\
&G_3=a_4
\end{aligned}
  \right. \ .
\end{equation*}
By integrating the last equation of (\ref{pel}), we get
\begin{equation*}
H(x^1,x^2)=b_1\sin({k_0x^2})-b_2\cos({k_0x^2})+K(x^1),
\end{equation*}
where $K=K(x^1)$, which, by differentiating with respect to $x^1$ and replaced in the second one, imply that
\begin{equation*}
(f_2K)'=0,
\end{equation*}
and we obtain
\begin{equation*}
K(x^1)=\displaystyle \frac{c_9}{f_2(x^1)}, \ c_9\in \mathbb R \ \ \textrm{and} \ \ H(x^1,x^2)=b_1\sin({k_0x^2})-b_2\cos({k_0x^2})+\displaystyle \frac{c_9}{f_2(x^1)}.
\end{equation*}
In this case,
\begin{equation*}
\left\{
    \begin{aligned}
&V^1(x^2)=b_1\cos({k_0x^2})+b_2\sin({k_0x^2})\\
&V^2(x^1,x^2)=b_1\sin({k_0x^2})-b_2\cos({k_0x^2})+\displaystyle \frac{c_9}{f_2(x^1)}\\
&V^3=a_4
\end{aligned}
  \right. \ .
\end{equation*}
Finally, let us notice that for $f_2$ constant (let's say, $k_2\in \mathbb R\setminus\{0\}$), we have obtained at (A.1) and (B.1) the expressions of the component functions $V^1$, $V^2$ and $V^3$, therefore, a linear combination of the two solutions is a solution, too:
\begin{equation*}
\left\{
    \begin{aligned}
&V^1(x^2,x^3)=a_1x^2+a_2x^3+a_3, \ \ a_1,a_2,a_3\in \mathbb R\\
&V^2(x^2,x^3)=-a_1k_2F(x^1)-c_1k_2x^3+a_4, \ \ c_1,a_4\in \mathbb R\\
&V^3(x^1,x^2)=-a_2k_3F(x^1)+c_1k_3x^2+a_5, \ \ a_5\in \mathbb R
\end{aligned}
  \right. \ .
\end{equation*}
By a direct computation, we notice that the converse implication holds true, too, i.e., since the component functions $V^1$, $V^2$ and $V^3$ of a vector field $V$ satisfy (\ref{s2d0}), then $V$ is a Killing vector field.
\end{proof}

\begin{corollary}\label{cada}
If $f_1=f_1(x^1)$, $f_2=k_2\in \mathbb R\setminus\{0\}$, $f_3=k_3\in \mathbb R\setminus\{0\}$, then $V=\sum_{k=1}^3V^kE_k$ is a Killing vector field if and only if
\begin{equation*}
\left\{
    \begin{aligned}
&V^1(x^2,x^3)=c_1x^2+c_2x^3+c_3\\
&V^2(x^1,x^3)=-c_1k_2F(x^1)-\frac{\displaystyle c_4}{\displaystyle k_3}x^3+c_5\\
&V^3(x^1,x^2)=-c_2k_3F(x^1)+\frac{\displaystyle c_4}{\displaystyle k_2}x^2+c_6
\end{aligned}
  \right. \ ,
\end{equation*}
where $F'=\displaystyle \frac{1}{f_1}$ and $c_1,c_2,c_3,c_4,c_5,c_6\in \mathbb R$.
\end{corollary}
\begin{proof}
It follows from Theorem \ref{te1}.
\end{proof}

\begin{corollary}
If $f_1=k_1\in \mathbb R\setminus\{0\}$, $f_2=f_2(x^1)$, $f_3=k_3\in \mathbb R\setminus\{0\}$, then $V=\sum_{k=1}^3V^kE_k$ is a Killing vector field if and only if one of the following assertions hold:
\hspace{0.5cm}\begin{equation*}
(i) \
\left\{
    \begin{aligned}
&V^1=0\\
&V^2(x^1)=\displaystyle\frac{c_1}{f_2(x^1)}, \ \ c_1\in \mathbb R\\
&V^3=c_2, \ \ c_2\in \mathbb R
\end{aligned}
  \right. \ ;
\end{equation*}

\hspace{0.5cm}
\begin{equation*}
(ii) \
\left\{
    \begin{aligned}
&V^1(x^2,x^3)=c_1x^2+c_2x^3+c_3, \ \ c_1,c_2,c_3\in \mathbb R\\
&V^2(x^1,x^3)=-\displaystyle\frac{c_1k_2}{k_1}x^1-c_4k_2x^3+c_5, \ \ c_4,c_5\in \mathbb R\\
&V^3(x^1,x^2)=-\displaystyle\frac{c_2k_3}{k_1}x^1+c_4k_3x^2+c_6, \ \ c_6\in \mathbb R
\end{aligned}
  \right. \ ,
\end{equation*}
where $F'=\displaystyle\frac{1}{f_1}$ and $f_2=k_2$ (is constant);

\hspace{0.5cm}
\begin{equation*}
(iii) \
\left\{
    \begin{aligned}
&V^1(x^2)=c_1x^2+c_2, \ \ c_1, c_2\in \mathbb R\\
&V^2(x^1,x^2)=k_1\left(\frac{f_2'}{f_2^2}\right)(x^1)\left[\displaystyle\frac{c_1}{2}(x^2)^2+c_2x^2+c_3\right]+\displaystyle \frac{c_1F_0(x^1)+c_4}{f_2(x^1)}, \ \ c_3, c_4\in \mathbb R\\
&V^3=c_3
\end{aligned}
  \right. \ ,
\end{equation*}
where $F_0'=-\displaystyle\frac{f_2^2}{k_1}$ and $f_2(x^1)=a_1e^{a_2x^1}$, $a_1, a_2\in \mathbb R\setminus \{0\}$;

\hspace{0.5cm}
\begin{equation*}
(iv) \
\left\{
    \begin{aligned}
&V^1(x^2)=c_1\cos({\sqrt{k}x^2})+c_2\sin({\sqrt{k}x^2}), \ \ c_1, c_2\in \mathbb R\\
&V^2(x^1,x^2)=k_1\left(\frac{f_2'}{f_2^2}\right)(x^1)\left[\displaystyle\frac{1}{\sqrt{k}}[c_1\sin({\sqrt{k}x^2})-c_2\cos({\sqrt{k}x^2})]+c_3\right]\\
&\hspace{72pt}+\displaystyle \frac{c_3kF_0(x^1)+c_4}{f_2(x^1)}, \ \ c_3, c_4\in \mathbb R\\
&V^3=c_3
\end{aligned}
  \right. \ ,
\end{equation*}
with $k:=\displaystyle\frac{k_1^2}{f_2^2}\cdot\left(\displaystyle\frac{f_2'}{f_2}\right)'\in (0,+\infty)$, where $F_0'=-\displaystyle\frac{f_2^2}{k_1}$;

\hspace{0.5cm}
\begin{equation*}
(v) \
\left\{
    \begin{aligned}
&V^1(x^2)=c_1e^{\sqrt{-k}x^2}+c_2e^{-\sqrt{-k}x^2}, \ \ c_1, c_2\in \mathbb R\\
&V^2(x^1,x^2)=k_1\left(\frac{f_2'}{f_2^2}\right)(x^1)\left[\displaystyle\frac{1}{\sqrt{-k}}\left(c_1e^{\sqrt{-k}x^2}-c_2e^{-\sqrt{-k}x^2}\right)+c_3\right]\\
&\hspace{72pt}+\displaystyle \frac{c_3kF_0(x^1)+c_4}{f_2(x^1)}, \ \ c_3, c_4\in \mathbb R\\
&V^3=c_3
\end{aligned}
  \right. \ ,
\end{equation*}
with $k:=\displaystyle\frac{k_1^2}{f_2^2}\cdot\left(\displaystyle\frac{f_2'}{f_2}\right)'\in (-\infty,0)$, where $F_0'=-\displaystyle\frac{f_2^2}{k_1}$.
\end{corollary}
\begin{proof}
It follows from Theorem \ref{te1}.
\end{proof}

\begin{corollary}
If $f_i=k_i\in \mathbb R\setminus \{0\}$ for any $i\in \{1,2,3\}$, then $V=\sum_{k=1}^3V^kE_k$ is a Killing vector field if and only if
\begin{equation*}
\left\{
    \begin{aligned}
&V^1(x^2,x^3)=-\frac{\displaystyle a_1}{\displaystyle k_2}x^2+\frac{\displaystyle a_2}{\displaystyle k_3}x^3+b_1\\
&V^2(x^1,x^3)=\frac{\displaystyle a_1}{\displaystyle k_1}x^1-\frac{\displaystyle a_3}{\displaystyle k_3}x^3+b_2\\
&V^3(x^1,x^2)=-\frac{\displaystyle a_2}{\displaystyle k_1}x^1+\frac{\displaystyle a_3}{\displaystyle k_2}x^2+b_3
    \end{aligned}
  \right. \ ,
\end{equation*}
where $a_1,a_2,a_3,b_1,b_2,b_3\in \mathbb R$.
\end{corollary}
\begin{proof}
It follows from Corollary \ref{cada}.
\end{proof}

\begin{example}
The vector field
$V=k_1^2(-k_3x^2+k_2x^3)\displaystyle \frac{\partial }{\partial x^1}+
k_2^2(k_3x^1-k_1x^3)\displaystyle \frac{\partial }{\partial x^2}+
k_3^2(-k_2x^1+k_1x^2)\displaystyle \frac{\partial }{\partial x^3}$
is a Killing vector field on
\begin{equation*}
\left(\mathbb R^3, \ g=\frac{1}{k_1^2}dx^1\otimes dx^1+\frac{1}{k_2^2}dx^2\otimes dx^2+\frac{1}{k_3^2}dx^3\otimes dx^3\right).
\end{equation*}
\end{example}

\begin{proposition}\label{pr1}
If $f_i=f_i(x^1)$ and $V^i=V^i(x^1)$ for any $i\in \{1,2,3\}$, then $V=\nolinebreak \sum_{k=1}^3V^kE_k$ is a Killing vector field if and only if one of the following assertions hold:

(i) $V^1=c_1\in \mathbb R\setminus \{0\}$, $V^2=c_2$, $V^3=c_3$, with $c_2,c_3\in \mathbb R$, and $f_2$, $f_3$ are constant;

(ii) $V^1=0$, $V^2=\frac{\displaystyle c_2}{\displaystyle f_2}$, $V^3=\frac{\displaystyle c_3}{\displaystyle f_3}$, with $c_2,c_3\in \mathbb R$.
\end{proposition}
\begin{proof}
In this case, \eqref{s2} becomes
\begin{equation*}
\left\{
    \begin{aligned}
&      (V^1)'=0\\
&V^1 f_2'=0\\
&V^1 f_3'=0\\
&(V^2)'=-V^2 \frac{\displaystyle f_2'}{\displaystyle f_2}\\
&(V^3)'=-V^3\frac{\displaystyle f_3'}{\displaystyle f_3}
    \end{aligned}
  \right. \ ,
\end{equation*}
from where we immediately get the conclusion.
\end{proof}

\begin{example}
The vector field $V=\displaystyle \frac{\partial }{\partial x^2}+\displaystyle \frac{\partial }{\partial x^3}$
is a Killing vector field on
\begin{equation*}
\left(\mathbb R^3, \ g=e^{x^1}dx^1\otimes dx^1+e^{2x^1}dx^2\otimes dx^2+e^{3x^1}dx^3\otimes dx^3\right).
\end{equation*}
\end{example}

\begin{remark}
If $f_1=f_2=:f(t)$, condition (B) from the proof of Theorem \ref{te1} becomes
\begin{equation*}
\left\{
    \begin{aligned}
&\left(\frac{\displaystyle f'}{\displaystyle f}\right)'+\left(\frac{\displaystyle f'}{\displaystyle f}\right)^2=k_1\in \mathbb R\\
&f\left(\frac{\displaystyle f'}{\displaystyle f^2}\right)'=k_2\in \mathbb R
    \end{aligned}
  \right. \ \ \Longleftrightarrow
\left\{
    \begin{aligned}
&\frac{\displaystyle f''}{\displaystyle f}=k_1\in \mathbb R\\
&\frac{\displaystyle f''}{\displaystyle f}-2\left(\frac{\displaystyle f'}{\displaystyle f}\right)^2=k_2\in \mathbb R
    \end{aligned}
  \right. \ .
\end{equation*}
Then $0\leq \left(\displaystyle\frac{\displaystyle f'}{\displaystyle f}\right)^2=\displaystyle\frac{\displaystyle k_1-k_2}{\displaystyle 2}$ and we obtain
$k_1\geq k_2$ and $f(t)=k_0e^{\sqrt{\frac{k_1-k_2}{2}}t}$, $k_0\in \mathbb R\setminus \{0\}$.
\end{remark}

\begin{example}
Any nowhere zero smooth function $f_1$ together with the function $f_2=c_1e^{c_2F}$, where $c_1\in \mathbb R\setminus \{0\}$, $c_2\in \mathbb R$, and $F'=\frac{\displaystyle 1}{\displaystyle f_1}$ satisfy the condition:
\begin{equation*}
\displaystyle\frac{f_1'}{f_1}\cdot\displaystyle\frac{f_2'}{f_2}+\left(\displaystyle\frac{f_2'}{f_2}\right)'=0
\end{equation*}
from (iv) of Theorem \ref{te1}.
\end{example}

\begin{example}
Functions $f_1$ and $f_2$ for which
\begin{equation*}
\left(\displaystyle\frac{f_1}{f_2}\right)^2\left[\displaystyle\frac{f_1'}{f_1}\cdot\displaystyle\frac{f_2'}{f_2}+\left(\displaystyle\frac{f_2'}{f_2}\right)'\right]
\end{equation*}
is constant are
\begin{equation*}
f_1(t)=a_1e^{bt} \ \ \textrm{and} \ \ f_2(t)=a_2e^{bt}, \ \ a_1,a_2 \in \mathbb R\setminus \{0\}, b\in \mathbb R.
\end{equation*}
\end{example}

\begin{example}
On the other hand, functions $f_1$ and $f_2$ for which
\begin{equation*}
\left(\displaystyle\frac{f_1}{f_2}\right)^2\left[\displaystyle\frac{f_1'}{f_1}\cdot\displaystyle\frac{f_2'}{f_2}+\left(\displaystyle\frac{f_2'}{f_2}\right)'\right]
\end{equation*}
is nonconstant are
\begin{equation*}
f_1(t)=a_1e^{2bt} \ \ \textrm{and} \ \ f_2(t)=a_2e^{bt}, \ \ a_1,a_2 \in \mathbb R\setminus \{0\}, b\in \mathbb R.
\end{equation*}
\end{example}

\begin{lemma}
If $f_i=f_i(x^i)$ for any $i\in \{1,2,3\}$, then $V=\sum_{k=1}^3V^kE_k$ is a Killing vector field if and only if
\begin{equation}\label{s4}
\left\{
    \begin{aligned}
&      \frac{\displaystyle \partial V^1}{\displaystyle \partial x^1}=0\\
&\frac{\displaystyle \partial V^2}{\displaystyle \partial x^2}=0\\
&\frac{\displaystyle \partial V^3}{\displaystyle \partial x^3}=0\\
&f_1\frac{\displaystyle \partial V^2}{\displaystyle \partial x^1}+f_2\frac{\displaystyle \partial V^1}{\displaystyle \partial x^2}=0\\
&f_2\frac{\displaystyle \partial V^3}{\displaystyle \partial x^2}+f_3\frac{\displaystyle \partial V^2}{\displaystyle \partial x^3}=0\\
&f_3\frac{\displaystyle \partial V^1}{\displaystyle \partial x^3}+f_1\frac{\displaystyle \partial V^3}{\displaystyle \partial x^1}=0
    \end{aligned}
  \right. \ .
\end{equation}
\end{lemma}
\begin{proof}
It follows immediately from \eqref{s1}.
\end{proof}

\begin{theorem}
If $f_1=f_1(x^1)$, $f_2=f_2(x^2)$, $f_3=k_3\in \mathbb R\setminus\{0\}$, then $V=\sum_{k=1}^3V^kE_k$ is a Killing vector field if and only if
\begin{equation*}
\left\{
    \begin{aligned}
&V^1(x^2,x^3)=-cF_2(x^2)+a_1x^3+a_2\\
&V^2(x^1,x^3)=cF_1(x^1)+b_1x^3+b_2\\
&V^3(x^1,x^2)=-a_1k_3F_1(x^1)-b_1k_3F_2(x^2)+b_3
\end{aligned}
  \right. \ ,
\end{equation*}
where $F_1'=\displaystyle\frac{1}{f_1}$, $F_2'=\displaystyle\frac{1}{f_2}$, and $a_1,a_2,b_1,b_2, b_3, c\in \mathbb R$.
\end{theorem}
\begin{proof}
In this case, we have
\begin{equation}\label{hy}
\left\{
    \begin{aligned}
&      \frac{\displaystyle \partial V^1}{\displaystyle \partial x^1}=0\\
&\frac{\displaystyle \partial V^2}{\displaystyle \partial x^2}=0\\
&\frac{\displaystyle \partial V^3}{\displaystyle \partial x^3}=0\\
&f_1\frac{\displaystyle \partial V^2}{\displaystyle \partial x^1}+f_2\frac{\displaystyle \partial V^1}{\displaystyle \partial x^2}=0\\
&f_2\frac{\displaystyle \partial V^3}{\displaystyle \partial x^2}+k_3\frac{\displaystyle \partial V^2}{\displaystyle \partial x^3}=0\\
&k_3\frac{\displaystyle \partial V^1}{\displaystyle \partial x^3}+f_1\frac{\displaystyle \partial V^3}{\displaystyle \partial x^1}=0
    \end{aligned}
  \right. \ .
\end{equation}
From the first three equations of (\ref{hy}), we get that
\begin{equation*}
V^1=V^1(x^2,x^3), \ \ V^2=V^2(x^1,x^3), \ \ V^3=V^3(x^1,x^2),
\end{equation*}
thus, from the fourth equation, we deduce that
\begin{equation*}
\left\{
    \begin{aligned}
&\frac{\displaystyle \partial V^1}{\displaystyle \partial x^2}=-\frac{\displaystyle 1}{\displaystyle f_2}F_{12}\\
&\frac{\displaystyle \partial V^2}{\displaystyle \partial x^1}=\frac{\displaystyle 1}{\displaystyle f_1}F_{12}
   \end{aligned}
  \right. \ ,
\end{equation*}
where $F_{12}=F_{12}(x^3)$, which, by integration, give
\begin{equation*}
\left\{
    \begin{aligned}
&V^1(x^2,x^3)=-F_{12}(x^3)F_2(x^2)+G_1(x^3)\\
&V^2(x^1,x^3)=F_{12}(x^3)F_1(x^1)+G_2(x^3)
\end{aligned}
  \right. \ ,
\end{equation*}
where $F_1'=\displaystyle\frac{1}{f_1}$, $F_2'=\displaystyle\frac{1}{f_2}$, $G_1=G_1(x^3)$ and $G_2=G_2(x^3)$.
Then
\begin{equation*}
\left\{
    \begin{aligned}
&f_1(x^1)\frac{\displaystyle \partial V^3}{\displaystyle \partial x^1}(x^1,x^2)=k_3F'_{12}(x^3)F_2(x^2)-k_3G'_1(x^3)\\
&f_2(x^2)\frac{\displaystyle \partial V^3}{\displaystyle \partial x^2}(x^1,x^2)=-k_3F'_{12}(x^3)F_1(x^1)-k_3G'_2(x^3)
\end{aligned}
  \right. \ .
\end{equation*}
By differentiating the above equations with respect to $x^3$, we get
\begin{equation*}
\left\{
    \begin{aligned}
&F''_{12}(x^3)F_2(x^2)=G''_1(x^3)\\
&-F''_{12}(x^3)F_1(x^1)=G''_2(x^3)
\end{aligned}
  \right. \ ,
\end{equation*}
and we conclude that $F''_{12}=0$, hence,
$G''_1=0=G''_2$.
Then
\begin{equation*}
\left\{
    \begin{aligned}
&F_{12}(x^3)=c_1x^3+c_2, \ \ c_1,c_2\in \mathbb R\\
&G_{1}(x^3)=a_1x^3+a_2, \ \ a_1,a_2\in \mathbb R\\
&G_{2}(x^3)=b_1x^3+b_2, \ \ b_1,b_2\in \mathbb R
\end{aligned}
  \right. \ ,
\end{equation*}
and we obtain
\begin{equation*}
\left\{
    \begin{aligned}
&\frac{\displaystyle \partial V^3}{\displaystyle \partial x^1}(x^1,x^2)=\frac{\displaystyle k_3c_1}{\displaystyle f_1(x^1)}F_2(x^2)-\frac{\displaystyle k_3a_1}{\displaystyle f_1(x^1)}\\
&\frac{\displaystyle \partial V^3}{\displaystyle \partial x^2}(x^1,x^2)=-\frac{\displaystyle k_3c_1}{\displaystyle f_2(x^2)}F_1(x^1)-\frac{\displaystyle k_3b_1}{\displaystyle f_2(x^2)}
\end{aligned}
  \right. \ ,
\end{equation*}
which, by integration, give
\begin{equation*}
\left\{
    \begin{aligned}
&V^3(x^1,x^2)=k_3c_1F_1(x^1)F_2(x^2)-k_3a_1F_1(x^1)+L_1(x^2)\\
&V^3(x^1,x^2)=-k_3c_1F_1(x^1)F_2(x^2)-k_3b_1F_2(x^2)+L_2(x^1)
\end{aligned}
  \right. \ ,
\end{equation*}
where $L_1=L_1(x^2)$ and $L_2=L_2(x^1)$.
Equating the above expressions and consequently differentiating the relation with respect to $x^1$ and $x^2$, we get
\begin{equation*}
\left\{
    \begin{aligned}
&\frac{\displaystyle 2k_3c_1}{\displaystyle f_1(x^1)}F_2(x^2)=\frac{\displaystyle k_3a_1}{\displaystyle f_1(x^1)}+L_2'(x^1)\\
&\frac{\displaystyle 2k_3c_1}{\displaystyle f_2(x^2)}F_1(x^1)=-\left[\frac{\displaystyle k_3b_1}{\displaystyle f_2(x^2)}+L_1'(x^2)\right]
\end{aligned}
  \right.  \ ;
\end{equation*}
therefore, $c_1=0$,
and further, we get
\begin{equation*}
L_1(x^2)+k_3b_1F_2(x^2)=L_2(x^1)+k_3a_1F_1(x^1),
\end{equation*}
which must be constant, let's say, $b_0$, and we obtain
\begin{equation*}
\left\{
    \begin{aligned}
&V^1(x^2,x^3)=-c_2F_2(x^2)+a_1x^3+a_2\\
&V^2(x^1,x^3)=c_2F_1(x^1)+b_1x^3+b_2\\
&V^3(x^1,x^2)=-k_3a_1F_1(x^1)-k_3b_1F_2(x^2)+b_0
\end{aligned}
  \right. \ .
\end{equation*}
By a direct computation, we notice that the converse implication holds true, too, i.e., since the component functions $V^1$, $V^2$ and $V^3$ of a vector field $V$ satisfy (\ref{hy}), then $V$ is a Killing vector field.
\end{proof}

\begin{example}
The vector field
$V=e^{x^1}(x^3-e^{x^2})\displaystyle \frac{\partial }{\partial x^1}+
e^{x^2}(x^3+e^{x^1})\displaystyle \frac{\partial }{\partial x^2}-
(e^{x^1}+e^{x^2})\displaystyle \frac{\partial }{\partial x^3}$
is a Killing vector field on
\begin{equation*}
\left(\mathbb R^3, \ g=e^{-2x^1}dx^1\otimes dx^1+e^{-2x^2}dx^2\otimes dx^2+dx^3\otimes dx^3\right).
\end{equation*}
\end{example}

\begin{proposition}\label{pr2}
If $f_i=f_i(x^i)$ and $V^i=V^i(x^i)$ for any $i\in \{1,2,3\}$, then $V=\nolinebreak \sum_{k=1}^3V^kE_k$ is a Killing vector field if and only if $V^i=c_i\in \mathbb R$ for $i\in \{1,2,3\}$.
\end{proposition}
\begin{proof}
In this case, \eqref{s4} becomes
\begin{equation*}
\left\{
    \begin{aligned}
&      (V^1)'=0\\
&(V^2)'=0\\
&(V^3)'=0
\end{aligned}
  \right. \ ,
\end{equation*}
from where we immediately get the conclusion.
\end{proof}

\begin{example}
The vector field
$V=e^{x^1}\displaystyle \frac{\partial }{\partial x^1}+
e^{x^2}\displaystyle \frac{\partial }{\partial x^2}+
e^{x^3}\displaystyle \frac{\partial }{\partial x^3}$
is a Killing vector field on
\begin{equation*}
\left(\mathbb R^3, \ g=e^{-2x^1}dx^1\otimes dx^1+e^{-2x^2}dx^2\otimes dx^2+e^{-2x^3}dx^3\otimes dx^3\right).
\end{equation*}
\end{example}

According to the previous results, we can state:

\begin{proposition} Let $V_1=\sum_{i=1}^3V_1^iE_i$ and $V_2=\sum_{i=1}^3V_2^iE_i$.

(i) If $f_i=f_i(x^1)$ and $V_k^i=V_k^i(x^1)$ for any $i\in \{1,2,3\}$ and $k\in\{1,2\}$, and $V_1$ and $V_2$ are Killing vector fields, then
\begin{equation*}
V_1^1=V_2^1+\tilde{c_1}, \ \ V_1^i=V_2^i+\tilde{c_i}+\frac{c_i}{f_i} \ \ (c_i\in \mathbb R, i\in \{2,3\}, \tilde{c_i}\in \mathbb R, i\in \{1,2,3\}).
\end{equation*}

(ii) If $f_i=f_i(x^i)$ and $V_k^i=V_k^i(x^i)$ for any $i\in \{1,2,3\}$ and $k\in\{1,2\}$, and $V_1$ and $V_2$ are Killing vector fields, then
\begin{equation*}
V_1^i=V_2^i+c_i \ \ (c_i\in \mathbb R, i\in \{1,2,3\}).
\end{equation*}
\end{proposition}
\begin{proof}
The assertions follow from Propositions \ref{pr1} and \ref{pr2}.
\end{proof}

\textit{Adara M. Blaga}

\textit{Department of Mathematics}

%\textit{Faculty of Mathematics and Computer Science}

\textit{West University of Timi\c{s}oara}

\textit{Bd. V. P\^{a}rvan 4, 300223, Timi\c{s}oara, Romania}

\textit{adarablaga@yahoo.com}
\end{document}